\documentclass[12pt]{amsart}
\usepackage{amssymb}
\usepackage[centertags]{amsmath}
\usepackage{amsfonts}
\usepackage{amsthm}
\usepackage{enumerate} 
\usepackage{xcolor}
\usepackage[margin=1in]{geometry}
\usepackage{ulem}
\usepackage{centernot}
\usepackage{csquotes}

\newtheorem{theorem}{Theorem}[section]
\newtheorem*{theorem*}{Theorem}

\newtheorem{corollary}[theorem]{Corollary}
\newtheorem{lemma}[theorem]{Lemma}
\newtheorem{remark}[theorem]{Remark}

\newtheorem{proposition}[theorem]{Proposition}

\theoremstyle{definition}

\newcommand{\nn}{\mathbb{N}}

\newcommand{\supp}{\mathrm{supp}}

\newcommand{\sspan}{\mathrm{span}}

\newcommand{\Isom}{\text{Isom}}
\newcommand{\Id}{\text{Id}}
\newcommand{\sign}{\text{sign}}

\begin{document}

\title{On the geometry of higher order Schreier spaces}

\author{Leandro Antunes}
\author{Kevin Beanland}
\author{H\`ung Vi\d{\^e}t Chu}

\address{Departamento de Matem\'atica, Universidade Tecnol\'ogica Federal do Paran\'a, Campus Toledo \\
 85902-490 Toledo, PR \\ Brazil}
\email{leandroantunes@utfpr.edu.br}

\address{Department of Mathematics, Washington and Lee University, Lexington, VA 24450.}
\email{beanlandk@wlu.edu}
\email{chuh19@mail.wlu.edu}

\thanks{H.V. Chu is an undergraduate student at Washington \& Lee University. Some of the results of this paper are part of the summer research work done under the supervision of the second author. Beanland and Chu acknowledge the support of Washington \& Lee's Lenfest Summer Research Scholars program.}

\thanks{Leandro Antunes was financially supported by Coordena\c{c}\~ao de Aperfei\c{c}oamento de Pessoal de N\'ivel Superior - Brasil (CAPES) (Process PDSE-88881.189744/2018-01, Finance Code 001) and by UTFPR (Process 23064.004102/2015-40).}



\begin{abstract}
For each countable ordinal $\alpha$ let $\mathcal{S}_{\alpha}$ be the Schreier set of order $\alpha$ and $X_{\mathcal{S}_\alpha}$ be the corresponding Schreier space of order $\alpha$. In this paper we prove several new properties of these spaces. 
\begin{enumerate}
    \item If $\alpha$ is non-zero then $X_{\mathcal{S}_\alpha}$ possesses the $\lambda$-property of R. Aron and R. Lohman and is a $(V)$-polyhedral spaces in the sense on V. Fonf and L. Vesely. 
    \item If $\alpha$ is non-zero and $1<p<\infty$ then the $p$-convexification $X^{p}_{\mathcal{S}_\alpha}$ possesses the uniform $\lambda$-property of R. Aron and R. Lohman. 
    \item For each countable ordinal $\alpha$ the space $X^*_{\mathcal{S}_\alpha}$ has the $\lambda$-property.
    \item For $n\in \mathbb{N}$, if $U:X_{\mathcal{S}_n}\to X_{\mathcal{S}_n}$ is an onto linear isometry then $Ue_i = \pm e_i$ for each $i \in \mathbb{N}$. Consequently, these spaces are light in the sense of Megrelishvili. 
\end{enumerate}
The fact that for non-zero $\alpha$, $X_{\mathcal{S}_\alpha}$ is $(V)$-polyhedral and has the $\lambda$-property implies that each $X_{\mathcal{S}_\alpha}$ is an example of space solving a problem of J. Lindenstrauss from 1966. The first example of such a space was given by C. De Bernardi in 2017 using a renorming of $c_0$. 
\end{abstract}

\maketitle
\section{Introduction}

The objective of this paper is to investigate several geometric properties of  higher order Schreier spaces, namely extreme points, $\lambda$-property, polyhedrality and isometries.

\subsection{Combinatorial Banach Spaces}
In \cite{Go-blog}, W.T. Gowers defines the combinatorial Banach space $X_\mathcal{F}$ as the completion of the vector space $c_{00}$ (finitely supported real scalar sequences)   with respect to the norm
$$\|x\|_{X_\mathcal{F}}=\sup\{\sum_{i\in F}|x(i)|: F\in\mathcal{F}\}, \quad x \in c_{00},$$
defined by a regular (i.e. compact, spreading and hereditary) family of finite subsets $\mathcal{F}$ of $\nn$ containing the singletons. 
A famous example of a regular family is $\mathcal{S}_1=\{F \subset \mathbb{N} :|F|\le \min F\}$ (here $|F|$ is the cardinality of $F$), and the combinatorial Banach space $X_{\mathcal{S}_1}$ is Schreier's space.

In this paper, we focus mainly on the combinatorial Banach spaces defined using the transfinite Schreier sets $(\mathcal{S}_\alpha)_{\alpha <\omega_1}$ (defined in \cite{AlA-Dissertationes}) as well as their $p$-convexifications.

\subsection{Extreme points}
Let $C$ be a nonempty closed convex subset of a Banach space $X$. We say that $x_0 \in C$ is an extreme point of $C$ if $x_0$ does not lie in the interior of any closed line segment contained in $C$. We denote by $E(C)$ the set of extreme points of $C$ and, for notational simplicity, we denote by $E(X)$ the set of extreme points of the unit ball of $X$, $Ba(X)$. For example, it is not hard to see that $E(c_0) = \emptyset$, which in particular implies that $c_0$ is not isometrically isomorphic to the dual of any Banach space.

In \cite{BDHQ-preprint} the second author of the current paper together with N. Duncan, M. Holt and J. Quigley proved several results for  combinatorial Banach spaces and, in particular, showed that the set of extreme points of the unit ball of $X_\mathcal{F}$ is at most countable for every regular family $\mathcal{F}$. In section \ref{section-extreme}, we build on this work.

For $p \in (1,\infty)$ we give, in section \ref{section-extreme}, a characterization of the extreme points of the $p$-convexification, $X_{\mathcal{S}_\alpha}^p$, (Theorem \ref{only reflexive}). Besides their own interest, the results of that section will be used several times in the remaining of this paper, namely in the proofs of Theorem \ref{lambda}, Theorem \ref{th53} and Lemma \ref{isometry items}.

\subsection{$\lambda$-property} 
In \cite{AronLoh-Pacific}, R. Aron and R. Lohman introduced geometric properties for Banach spaces, called the $\lambda$-property and uniform $\lambda$-property.  A space $X$ is said to have the $\lambda$-property if for all $x\in Ba(X)$, there exists $0<\lambda\leqslant 1$ such that $x=\lambda e+(1-\lambda)y$ for some $e \in E(X)$, $y \in    Ba(X)$. A space $X$ is said to have the uniform $\lambda$-property if there exists $\lambda_0 > 0$ such that for every $x \in Ba(X)$, $\lambda_0 \leqslant \sup \{\lambda > 0; \exists \,e \in E(X), y \in Ba(X); x = \lambda e + (1-\lambda)y\}.$

These properties have been extensively studied by many authors over the past 25 years (e.g. \cite{AronLohSu-PAMS,Authors-Lambda1,Authors-Lambda2,Authors-Lambda5,Lin-Lambda,Authors-Lambda4}). In 1989 \cite{ShTr-Glasgow}, Th. Shura and D. Trautman proved that the Schreier space has the $\lambda-$property and the set of extreme points is countably infinite.  In Section \ref{section-lambda} we prove the following:
\begin{theorem*}
Let $\alpha$ be a countable and non-zero ordinal.
\begin{enumerate}
    \item For $p \in  (1,\infty)$, the space $X^p_{\mathcal{S}_\alpha}$ has the uniform $\lambda$-property.
    \item The space  $X_{\mathcal{S}_\alpha}$ has the $\lambda$-property.
\end{enumerate} 
\label{lambda1}
\end{theorem*}

We also give a characterization for the extreme points of $X_{\mathcal{S}_\alpha}^*$ for countable $\alpha$ (Proposition \ref{prop54}).

\subsection{Polyhedrality}

A Banach space $X$ is called polyhedral in \cite{Klee-Acta} if the unit ball of every finite dimensional subspace of $X$ is a polytope (i.e. has finitely many extreme points). Some important examples of polyhedral spaces are $c_0$ and $C(K)$ spaces for $K$ a countable, compact, Hausdorff space. V. Fonf \cite{Fonf-polyhedral} showed that a polyhedral space must be $c_0$-saturated (that is, every infinite dimensional subspace has a further subspace isomorphic to $c_0$). In addition, for each countable $\alpha<\omega_1$ the space $X_{\mathcal{S}_\alpha}$ embeds isometrically in a $C(K)$ for an appropriately chosen countable compact Hausdorff space $K$ (see, for example, \cite{CastGon-Extracta,Rosenthal-Handbook}).  Therefore each  $X_{\mathcal{S_\alpha}}$ is a polyhedral Banach space.

In a recent paper \cite{DeB-Israel}, C. De Bernardi presents a space $X$ that is a renorming of $c_0$ and that is both polyhedral and has the property that $Ba(X)$ is the closed convex hull of its extreme points. The existence of a space with these properties solved a problem of J. Lindenstrauss from 1966 \cite{Lin-Klee}. 

De Bernardi also observes that his space has the following property called $(V)$-polyhedral which is stronger than being polyhedral.  A Banach space $X$ is called a $(V)$-polyhedral space (Fonf and Vesel\'y  in \cite{FonfVes-poly}) if 
$$\sup\{f(x): f \in E(X^*) \setminus D(x)\}<1$$
for all $x \in S(X)$ (the unit sphere of $X$) where $D(x)=\{g \in S(X^*): g(x)=1\}$. This was the fifth definition concerning polyhedrality in their paper, hence the notation $(V)$-polyhedral.

We will prove in Theorem \ref{th53} that for each countable $\alpha$, $X_{\mathcal{S}_\alpha}$ is a $(V)$-polyhedral space. Moreover, $Ba(X_{\mathcal{S}_\alpha})$ is the closed convex hull of its extreme points, i.e., $X_{\mathcal{S}_\alpha}$ are new solutions for Lindenstrauss' problem.

\subsection{Isometry group of $X_{\mathcal{S}_n}$}

Given a Banach space $X$, we denote by $\Isom(X)$ the group formed by all surjective linear isometries of $X$. The characterization of the isometries play a central role in the field of geometry of Banach spaces and can be found already in the famous Banach's treatise of 1932 \cite{Ba-book}, in which he gives the general form of isometries of classical spaces, such as $c$, $c_0$, $C(K)$, $\ell_p$ and $L_p, 1 \leq p < \infty$. Characterizations for other spaces can be found in \cite{fleming-isometries}.

In the final section of the paper, Section \ref{section-isometry}, we characterize $ \Isom(X_{\mathcal{S}_n})$ for each $n \in \mathbb{N}$.

As an application of this characterization, we classify the groups $\Isom(X_{\mathcal{S}_n})$, $n \in \mathbb{N}$ in terms of being light. In \cite{megrelishvili-light}, Megrelishvili defines the concept of light group of isomorphisms of a Banach space $X$ as follows: a group $G \leqslant GL(X)$  light if the Weak Operator Topology {(WOT)} and the Strong Operator Topology {(SOT)} coincide on $G$. He proves that every bounded group of isomorphisms of a Banach space with the Point of Continuity Property (PCP) (e.g., spaces with the Radon-Nikodym Property, including reflexive spaces, and separable dual spaces) is light. In \cite{AFGR-light}, the authors classify in terms of being light the isometry groups of several classical Banach spaces without PCP, such as $c_0, c, \ell_1, \ell_\infty, L_1[0,1]$ and $C(K)$, where $K$ is a infinite compact connected space. They also prove that if $X$ admits a locally uniformly convex renorming invariant under the action of a group $G \leqslant GL(X)$, then $G$ is light.

We prove in Proposition \ref{light} that $\Isom(X_{\mathcal{S}_n})$ is light, for every $n \in \mathbb{N}$. This provides new examples of light groups of a Banach space without PCP.

\subsection*{Acknowledgments} The authors would like to thank Ryan Causey for providing the proof of Proposition \ref{inf convex}. Our original proof was more complicated.



\section{Extreme points of higher order Schreier spaces}
\label{section-extreme}



Let $\mathcal{A}_n$ denote the set of finite subsets of $\mathbb{N}$ with cardinality less than $n$. The higher order Schreier families are defined  in \cite{AlA-Dissertationes} as follows. Letting $\mathcal{S}_0=\mathcal{A}_1$ and supposing that $\mathcal{S}_\alpha$ has been defined for some ordinal $\alpha<\omega_1$, we define
$$\mathcal{S}_{\alpha+1}=\{\cup^{n}_{i=1} E_i: \{\min E_i\}_{i=1}^n \in \mathcal{S}_1 \mbox{ and }E_i \in  \mathcal{S}_\alpha\}\cup \{\emptyset\}.$$
If $\alpha$ is a limit ordinal then we fix $\alpha_n\nearrow\alpha$ and define $\mathcal{S}_\alpha=\{\emptyset\} \cup\{F:\exists n \leqslant \min F,\,F\in\mathcal{S}_{\alpha_{n}}\}.$ We may assume (see for example \cite{Ca-Studia}), that for each $n \in \mathbb{N}$ we have $\mathcal{S}_{\alpha_{n}}\subset \mathcal{S}_{\alpha_{n+1}}$. For each $\alpha<\omega_1$ the set $\mathcal{S}_\alpha$ is a regular family. A set $F \in \mathcal{S}_\alpha$ is non-maximal if and only if for every $l > \max F$, $F \cup \{l\} \in  \mathcal{S}_\alpha$. We denote by $\mathcal{S}^{MAX}_\alpha$ the maximal $\mathcal{S}_\alpha$ sets. Many properties of the collection $(\mathcal{S}_\alpha)_{\alpha <\omega_1}$ can be found in \cite{ATol-Memoirs}. We will use the following general remarks concerning Schreier families of finite order. A good reference for properties of finite order Schreier families is \cite[Lemma 3.8]{GasLeung-Studia}. Recall that if $F,G$ are finite subsets of $\mathbb{N}$ then we say that $F=\{k_1, \ldots, k_n\}$ is a spread of $G=\{\ell_1,\ldots, \ell_m\}$ (written in increasing order) if $m=n$ and $\ell_i \leqslant k_i$ for each $1 \leqslant i \leqslant n$. In addition, we write $E < F $ if $\max E <\min F$.





\begin{remark}
Let $n \in \mathbb{N}$. We mention two facts about a maximal set in $\mathcal{S}^{MAX}_n$.
\begin{enumerate}
\item A set $E \in \mathcal{S}_n^{MAX}$ if and only if for each $m,k$ with $m+k=n$ there is a unique sequence $(E_i)_{i=1}^d$ so that $E=\cup_{i=1}^d E_i$ with $(\min E_i)_{i=1}^d \in \mathcal{S}_m^{MAX}$, $E_1<E_2<\ldots E_d$ are in $\mathcal{S}_k^{MAX}$. 
\item Let $n\in\mathbb{N}$ with $m+k = n$. If a set $G\in \mathcal{S}^{MAX}_n$ is written as $\cup_{i=0}^d G_i$, where $G_0 < G_1 < \ldots < G_d \in \mathcal{S}^{MAX}_{m}$, then $(\min G_i)_{i=0}^d\in \mathcal{S}^{MAX}_k$.
\end{enumerate}
\label{split em up}
\end{remark}

\begin{remark}
Suppose that $G \in \mathcal{S}_n^{MAX}$ and $F\subset \mathbb{N}$ with $\min G < \min F$, $F$ a spread of $G$ with $|F|=|G|$. Then if $j>\min G$, $\{j\} \cup F \in \mathcal{S}_n$. \label{squeeze in}
\end{remark}

\begin{proof}
By Lemma \ref{split em up}, we can write $G=\cup_{i=1}^d G_i$ so that $G_1< \cdots <G_d$ in $\mathcal{S}^{MAX}_{n-1}$, $(\min G_i)_{i=1}^d \in \mathcal{S}^{MAX}_1$, and $d=\min G_1$. Since $|F|=|G|$ and $F$ is a spread of $G$ there is a corresponding decomposition $F=\cup_{i=1}^d F_i$ where $F_i$ is a spread of $G_i$. Let $j >\min G$. Then $$\{\{j\},F_1,\ldots, F_d\}$$
is a collection of $d+1$-many $\mathcal{S}_{n-1}$ sets and the overall minimum is greater than or equal to $d+1$. Therefore $\{j\} \cup F \in \mathcal{S}_n$, as desired.
\end{proof}

Let $(e_i)_{i=1}^\infty$ and $(e_i^*)_{i=1}^\infty$ both denote the standard unit vector basis of $c_{00}$. The sequence $(e_i)_{i=1}^\infty$ is a $1$-unconditional Schauder basis for each of the following spaces. For each regular family $\mathcal{F}$ and $p\in (1, \infty)$, we denote the $p$-convexification of $X_\mathcal{F}$ by $X^{p}_\mathcal{F}$ (and for notation convenience $X_\mathcal{F}=X_\mathcal{F}^1$). The space $X^{p}_\mathcal{F}$ is the completion of $c_{00}$ with respect to the following norm:  $$\|x\|_{X^{p}_\mathcal{F}}=\sup_{F\in\mathcal{F}}(\sum_{i\in F}|x(i)|^p)^\frac{1}{p}.$$

We call $F \in \mathcal{F}$ a 1-set for $x \in S(X_{\mathcal{F}}^p)$ if  $(\sum_{i\in F}|x(i)|^p)^\frac{1}{p}=1$ and $x(i)\not=0$ for any $i \in \mathcal{F}$. Let $\mathcal{F}^1_x$ be the set of all 1-sets of $x$.  Let $\mathcal{A}_x =\{F \in \mathcal{F}: \sum_{i \in F}|x(i)|^p=1\}$. Note that if $F \in \mathcal{A}_x \setminus \mathcal{F}_x^1$ then there is a $G \subset F$ in $\mathcal{F}_x^1$ so that for $i \in F \setminus G$ , $x(i)=0$. Note that $x$ has only maximal $1$-sets if and only if $\mathcal{A}_x= \mathcal{F}_x^1$. 

In the next proposition, we prove that the set $\mathcal{F}_x^1$ is finite, for $x \in S(X^p_{\mathcal{S}_\alpha})$ and $0< \alpha < \omega_1$, and every extreme point of $X_{\mathcal{S}_\alpha}^p$ has finite support, for $0<\alpha < \omega_1$ and $1 \leqslant p < \infty$. This proposition will be used several times in this paper.

\begin{proposition}
\label{lots of items}
Let $\mathcal{F} \in \{ \mathcal{S}_\alpha : 0<\alpha < \omega_1\}$, $p \in [1,\infty)$ and $x \in S(X^p_{\mathcal{F}})$.  Then the following hold:
\begin{enumerate}
    \item The set $\mathcal{F}_x^1$ is finite.
    \item There is an $\varepsilon_x>0$ (which we call the $\varepsilon$-gap for $x$) so that each $F \in \mathcal{F} \setminus \mathcal{A}_x$, $\sum_{i \in F} |x(i)|^p < 1-\varepsilon_x$.
    \item $E(X_{\mathcal{F}}) \subset c_{00}$
\end{enumerate}
\end{proposition}

\begin{proof}
The case of $p=1$ in the above proposition is proved in \cite{BDHQ-preprint}. For a vector $x= \sum_i x(i) e_i$ define $x^p = \sum_i |x(i)|^p e_i$. Observe that if $\|\sum_i x(i) e_i\|_{X_{\mathcal{F}}^p} = 1$ then $\|\sum_i |x(i) |^p e_i\|_{X_{\mathcal{F}}} = 1$. Using  \cite[Lemma 2.5]{BDHQ-preprint} we can find $\varepsilon_{x^p}>0$ so that 
$$\sum_{i\in F} |x(i)|^p <1- \varepsilon_{x^p}$$
for all $F\in \mathcal{F} \setminus \mathcal{A}_{x^p}$. Note that $F\in \mathcal{A}_{x}$ for $x \in X_{\mathcal{F}}^p$ if and only if $F \in \mathcal{A}_{x^p}$ for $x^p \in X_{\mathcal{F}}$. This proves the first two claims.

Suppose that $x \in S(X_{\mathcal{F}}^p)\setminus c_{00}$. Let $k$ with $x(k)\not=0$ be larger than the maximum of every $F \in \mathcal{F}_x^1$. Note it is not possible for $F\cup\{k\} \in {\mathcal{F}}$ for any $F \in \mathcal{F}_x^1$. That is, $\mathcal{F}_x^1$ consists of only maximal sets. Therefore if we consider $F \in {\mathcal{F}}$ that contains $k$ then $F\not \in \mathcal{F}_x^1$ and so $$(\sum_{i \in F}|x(i)|^p)^{1/p}< (1-\varepsilon_x)^{1/p} \leqslant 1-\varepsilon_x/p.$$
We can therefore perturb $x(k)$ by a value less than $\varepsilon_x/p$ to produce $y,z \in S(X_{\mathcal{F}}^p)$ with $x=1/2(y+z)$. This is the desired result. 
\end{proof}

In Theorem \ref{only reflexive} we will give a characterization for the extreme points of $E(X_{\mathcal{S}_\alpha}^p), 0<\alpha < \omega_1, 1 < p < \infty$. This is the main result of this section. In the proof of the theorem, we will need to use a few decompositions of the points $x \in S(X_{\mathcal{S}_\alpha}^p)$, given by Lemma \ref{the big one}.

The proof of Lemma \ref{the big one} uses the next result that  follows from the significantly stronger statement in \cite[Proposition 12.9]{ATol-Memoirs}.

\begin{proposition}
Fix ordinals $\eta<\alpha <\omega_1$ and $p\in [1,\infty)$. For each $\varepsilon >0$ and $n\in \mathbb{N}$ there exist $F \in \mathcal{S}_\alpha^{MAX}$ with $n \leqslant \min F$ and a sequence non-negative of scalars $(a_i)_{i \in F}$ with $\sum_{i \in F}a_i^p =1$ so that for each $G \in \mathcal{S}_\eta$,
$ \sum_{i \in G} a_i^p <\varepsilon.$
\label{RAA}
\end{proposition}

\begin{lemma}
Let $\mathcal{F} \in \{\mathcal{S}_\alpha : 0 < \alpha <\omega_1\}$, $p \in [1,\infty)$ and $x \in S(X^p_{\mathcal{F}})$.  Then the following hold:
\begin{enumerate}
    \item There exist $x_1,x_2 \in S(X^p_{\mathcal{F}})$ with $x_1 \in c_{00}$ and $x=\frac{1}{2}(x_1+x_2)$.
    \item If $x \in c_{00}$, there exist $x_1,x_2 \in S(X^p_{\mathcal{F}})\cap c_{00}$ so that both $x_1$ and $x_2$ have non-maximal 1-sets and $x=\frac{1}{2}(x_1+x_2)$. 
    \item If $x\in c_{00}$, there exist $x_1,x_2 \in S(X^p_{\mathcal{F}})\cap c_{00}$ so that $x=\frac{1}{2}(x_1 +x_2)$ and for each $i \leqslant \max \supp ~x_1$  there is an $F \in \mathcal{A}_{x_1}$ with $i \in F$. 
\end{enumerate}
\label{the big one}
\end{lemma}

\begin{proof}
We first prove item (1). Let $x \in S(X_{\mathcal{F}}^p)$. Using Proposition \ref{lots of items} we can find $\varepsilon_x>0$ (the $\varepsilon$-gap for $x$). Fix $N \in \mathbb{N}$ so that $\|\sum_{i > N} x(i) e_i\|^p <\varepsilon_x/2$ and $N > \max\{\max F : F \in \mathcal{F}_x^1\}$. Let $x_1 = \sum_{i=1}^N x(i) e_i$ and $x_2 = 2x-x_1$. Clearly $\|x_1 \|\leqslant \|x\|=1$. It suffices to prove that $\|x_2\|\leqslant 1$. Suppose first that $F \in \mathcal{A}_x$. Then there is a $G \subset F$ with $G \in \mathcal{F}_x^1$. Then $\max G < N$, and so 
$$(\sum_{i \in F} |x_2(i)|^p)^{1/p} = (\sum_{i \in G} |x(i)|^p)^{1/p}=1$$
If $F \in \mathcal{F}\setminus \mathcal{A}_x$ we have the following:
$$(\sum_{i \in F} |x_2(i)|^p)^{1/p}=(\sum_{i \in F, i \leqslant N} |x(i)|^p + 2\sum_{i \in F, i > N} |x(i)|^p)^{1/p}< (1 -\varepsilon_x + \varepsilon_x)^{1/p} \leqslant 1. $$
Hence, $x_1, x_2 \in Ba(X_{\mathcal{F}}^p)$, and since $x=\frac{1}{2}(x_1+x_2)$ and $\|x\| = 1$, we must have $x_1, x_2 \in S(X_{\mathcal{F}}^p)$. This finishes the proof of item $(1)$.

Let us prove item $(2)$. We may assume that $x \in c_{00}$ has only maximal $1$-sets and let $N=\max \supp \,x$. Recall that $\mathcal{F}=\mathcal{S}_\alpha$ for some ordinal $0<\alpha <\omega_1$. We must distinguish between the cases that $\alpha$ is a successor and limit ordinal. In both cases we apply Proposition \ref{RAA}. In the limit case we apply this Proposition for $\eta=\alpha_N$ and in the successor case, for $\eta$ with $\eta+1 =\alpha$. Using Proposition \ref{RAA} we can find $A \in \mathcal{S}_\alpha^{MAX}$ with $\min A >N$ and convex scalars $(a_i)_{i \in A}$ so that for all $G \in \mathcal{S}_\eta$
$$\sum_{i \in G} a^p_i <\frac{\varepsilon_x}{2N}. $$
Let $i_0 = \max A$ and $F_0 = A \setminus \{i_0\}$ and $b^p_i = a_i^p/(1-a^p_{i_0})$ for $i \in F_0$. Clearly $(b^p_i)_{i \in F_0}$ are convex scalars, $F_0$ is non-maximal and if $G \in \mathcal{S}_\eta$,
$$\sum_{i \in G} b^p_i <\frac{\varepsilon_x}{N}. $$
Let $x_1 = x+ \sum_{i\in F_0} b_i e_i$ and $x_2= x- \sum_{i\in F_0} b_i e_i$. Since $x_1$ and $x_2$ both have $F_0$ as a non-maximal $1$-sets we are done once we can show that $\|x_1\|=\|x_2\| =1$. In the first case we assume $F \in \mathcal{A}_x$. Since $x$ has only maximal 1-sets we know  that $x(i)\not=0$ for all $i \in F$ therefore $F \subset \supp\, x$ and 
$$\sum_{i \in F} |x_1(i)|^p=\sum_{i \in F}|x(i)|^p =1.$$
Now suppose that $F \in \mathcal{F} \setminus \mathcal{A}_x$. In the case that $\alpha$ is a limit ordinal we have the following argument: If $\min F > \max \supp\,x$ the $\sum_{i \in F} |x_1(i)|^p \leqslant \sum_{i \in F} b_i^p \leqslant 1$. Therefore we assume $\min F \leqslant \max \supp \,x = N$. By definition of $\mathcal{S}_\alpha$ for $\alpha$ a limit ordinal we have $F \in \mathcal{S}_{\alpha_{\min F}}\subset \mathcal{S}_{\alpha_{N}}$. Then
$$\sum_{i \in F} |x_1(i)|^p = \sum_{i \in F~ i \leqslant N}|x(i)|^p + \sum_{i \in F~ i > N}b^p_i < 1-\varepsilon_x + \frac{\varepsilon_x}{N} <1.$$
This concludes the limit ordinal case.

Now we consider the case that $\alpha =\eta +1$. Again we may assume that $\min F \leqslant \max\supp~ x$. We know that, by definition,  $F=\cup_{i=1}^d F_i$ where $F_1< F_2 <\cdots < F_d$ and $F_i \in S_\eta$ and $d \leqslant \max \supp ~x \leqslant N$. Consider the following estimate.
\begin{equation*}
    \begin{split}
      \sum_{i \in F} |x_1(i)|^p = 1- \varepsilon_x + \sum_{i \in F~i >N} b_i^p \leqslant 1- \varepsilon_x + \sum_{j=1}^d\sum_{i \in F_j~i >N} b_i^p < 1-\varepsilon_x + d\, \frac{\varepsilon_x}{N} <1.
    \end{split}
\end{equation*}
This shows that $\|x_1\|\leqslant 1$. The same proof yields $\|x_2\|\leqslant 1$, as desired.  Again, since $x=\frac{1}{2}(x_1+x_2)$ and $\|x\| = 1$, we must have $x_1, x_2 \in S(X_{\mathcal{F}}^p)$.

Finally, we prove item $(3)$ of the lemma. Let $x \in c_{00}$ and consider the following procedure: Let $i_1\in [1, \max \supp \,x]$ be minimum so that for all $F \in {\mathcal{F}}$, with $i_1 \in F$, $\sum_{i \in F} |x(i)|<1$. If no such $i_1$ exists we are done (let $x=x_1=x_2$). 
Since there are only finitely many $F \in \mathcal{F}$ containing $i_1$ with $\max F \leqslant \max\supp\, x$ we can find $F_1 \in \mathcal{F}$ with 
$$(\sum_{i \in F_1} |x(i)|^p)^{1/p}= \sup\{ (\sum_{i \in F} |x(i)|^p)^{1/p}: F \in \mathcal{F}, ~i_1 \in F\}.$$
Find $\delta_{i_1} >0$ so that
$$|x(i_1) + \textrm{sign}(x(i_1))\delta_{i_1}|^p + \sum_{i \in F_1, i\not=i_1} |x(i)|^p= 1 $$
Let $x_{1,1} = x+\textrm{sign}(x(i_1))\delta_{i_1} e_{i_1}$ and $x_{2,1} = x- \textrm{sign}(x(i_1))\delta_{i_1}e_{i_1}$. We shall prove that $\|x_{1,1}\|\leqslant 1$. As such we must show for each $F \in \mathcal{F}$, $\sum_{i \in F}|x_{1,1}(i)|\leqslant 1$. The case that $F\in \mathcal{F}$ and does not contain $i_1$ it follows from the fact that $\|x\|\leqslant 1$ and so we assume $i_1 \in F$. In this case, we use the definition of $F_1$ to observe that

\begin{equation*}
\begin{split}
   \sum_{i \in F}|x_{1,1}(i)|^p & =|x(i_1) + \textrm{sign}(x(i_1))\delta_{i_1}|^p + \sum_{i \in F, i\not=i_1} |x(i)|^p \\ & \leqslant |x(i_1) + \textrm{sign}(x(i_1))\delta_{i_1}|^p + \sum_{i \in F_1, i\not=i_1} |x(i)|^p =1. 
   \end{split}
\end{equation*}
Therefore $\|x_{1,1}\|\leqslant 1$. Since $|x_{2,1}(i_1)|\leqslant |x_{1,1}(i_1)|$ we have $\|x_{2,1}\|\leqslant 1$ and by the same reasons as the previous items, we conclude that $\|x_{1,1}\| = \|x_{2,1}\| = 1$ and also, trivially, that $x=\frac{1}{2}(x_{1,1}+x_{2,1})$.
In order to produce a vector satisfying the claim we inductively apply the above procedure as follows: Find the minimum $i_2> i_1$ in $[1,\max\supp \,x]$ and so that for all $F \in \mathcal{F}$, with $i_2 \in F$, $\sum_{i \in F} |x(i)|<1$. If no such $i_2$ exists we are done. Since there are only finitely many $F \in \mathcal{F}$ containing $i_2$ with $\max F \leqslant \max\supp\, x$ we can find $F_2 \in \mathcal{F}$ with 
$$(\sum_{i \in F_2} |x_{1,1}(i)|^p)^{1/p}= \sup\{ (\sum_{i \in F} |x_{1,1}(i)|^p)^{1/p}: F \in \mathcal{F}, ~i_2 \in F\}.$$
Find $\delta_{i_2} >0$ so that
$$|x_{1,1}(i_2) + \textrm{sign}(x_{1,1}(i_2))\delta_{i_1}|^p + \sum_{i \in F_2, i\not=i_1} |x_{1,1}(i)|^p= 1 $$
Let $x_{1,2} = x_{1,1}+\textrm{sign}(x(i_2))\delta_{i_2} e_{i_2}$ and $x_{2,2} = x_{1,2}- \textrm{sign}(x(i_2))\delta_{i_2}e_{i_2}$. Arguing as before we have $\|x_{1,2}\|\leqslant 1, \|x_{2,2}\| \leqslant 1$ and $x=\frac{1}{2}(x_{1,2}+x_{2,2})$. This procedure can be iterated finitely many times to exhaust $\supp \, x$ in order to produce for some $n \in \mathbb{N}$ $x_{1,n}$ and $x_{2,n}$ with $\|x_{1,n}\|\leqslant 1, \|x_{2,n}\| \leqslant 1$ and $x=\frac{1}{2}(x_{1,n}+x_{2,n})$ so that $x_{1,n}$ has the property that for each $i\leqslant \max \supp x_{1,n}$ there is an $F \in \mathcal{A}_{x_{1,n}}$ with $i \in F$. This yields the desired decomposition.
\end{proof}

 The next theorem is our main result in this section. It provides a characterization of extreme points in $Ba(X_{\mathcal{F}}^p)$ and $p\in (1,\infty)$. Such a characterization will be used to prove, in the next section, that the space $X^p_{\mathcal{S}_\alpha}$ has the uniform $\lambda$-property, for $1 < p < \infty$ and $0<\alpha < \omega_1$.

\begin{theorem}
\label{only reflexive}
Let $\mathcal{F} \in \{\mathcal{S}_\alpha : 0 < \alpha <\omega_1\}$, $p\in (1,\infty)$  and $x \in S(X_{\mathcal{F}}^p)$. Then $x \in E(X_{\mathcal{F}}^p)$ if and only if $x \in c_{00}$, $\mathcal{A}_x$ has a non-maximal set and for all $i \leqslant \max \supp\,x$ there is an $F \in \mathcal{A}_x$ with $i \in F$. Moreover if $p=1$ then the forward implication holds.
\end{theorem}

\begin{proof}
We first prove the reverse implication. Suppose $x \in c_{00}$ and satisfies the assumptions. Let $x=1/2(z+y)$ and $F \in$  $\mathcal{A}_x$. Then $\sum_{i \in F} |x(i)|^p=1$. Since every element of the sphere of $\ell_p^{|F|}$ is an extreme point, we know in order for $\sum_{i \in F} |y(i)|^p=\sum_{i \in F} |z(i)|^p=1$ we must have $x(i)=y(i)=z(i)$  for all $i \in F$. Our assumption is that all $i \leqslant \max \supp \,x$ are contained in a set $F \in \mathcal{A}_x$. Therefore $x(i)=y(i)=z(i)$ for all such $i \leqslant \max \supp \,x$. Now let $i> \max \supp \,x$. Find a non-maximal $F \in \mathcal{A}_x$ with $\max F \leqslant \max \supp \,x$. Then $F\cup\{i\}\in \mathcal{A}_x$ and consequently $x(i)=y(i)=z(i)$ or else we we could sum over $F\cup\{i\}$ to show that either $y$ or $z$ had norm greater than $1$. Therefore $z=y=x$ which implies that $x \in E(X_{\mathcal{F}}^p)$.

We now prove the forward implication as well as the `moreover' statement. Let $x \in S(X_{\mathcal{F}}^p)$ for $p \in [1,\infty)$. First, Proposition \ref{lots of items} states that $E(X_{\mathcal{F}}^p)$ is a subset of $c_{00}$. We can assume that either every  set in $\mathcal{A}_x$ is maximal or there is an $i \leqslant \max \supp \,x$ not contained in any $F \in \mathcal{A}_x$. In the former case we have $\mathcal{A}_x = \mathcal{F}_x^1$ and since $\mathcal{F}_x^1$ is finite there is a $k>\max\{\max F : F \in \mathcal{F}_x^1\}$. We can perturb $x(k)$ by any value $\delta>0$ with $\delta<\varepsilon_x/p$ and create new vectors $y=x-\delta x(k)e_k$ and $z=x+\delta x(k)e_k$ that are in $S(X_{\mathcal{F}}^p)$ and satisfy $x=1/2(y+z)$. In the later case, we can find the coordinate $k \leqslant \max \supp ~x$ and similarly show that $x$ is not an extreme point.  
\end{proof}

\section{$\lambda$-property for Schreier spaces}
\label{section-lambda}

Recall from the introduction that a space $X$ is said to have the $\lambda$-property if for all $x\in Ba(X)$, there exists $0<\lambda\leqslant 1$ such that $x=\lambda e+(1-\lambda)y$ for some $e \in E(X)$, $y \in    Ba(X)$.

When a vector $x$ can be written in terms of $\lambda, e, y,$ we denote $(e,y,\lambda)\sim x.$ For a vector $x$, we may find different sets $(e,y,\lambda)$ such that $(e,y,\lambda)\sim x$. This leads Aron and Lohman \cite{AronLoh-Pacific} to define the following function: Given $x\in Ba(X)$,
$$\lambda(x)=\sup\{\lambda:(e,y,\lambda)\sim x\}.$$ 
If there exists $\lambda_0>0$ such that for all $x\in Ba(X),\lambda(x)\geqslant  \lambda_0$, we say that $X$ has the uniform $\lambda$-property. Note that for a non-zero $x \in Ba(X)$ we have
$$x = \frac{1}{2}\frac{x}{\|x\|} + \frac{1}{2}(2 \|x\| -1)\frac{x}{\|x\|}.$$

Consequently, in order to verify that $X$ has the $\lambda$-property it suffices to show that for each $x \in S(X)$ there are $(e,y,\lambda) \in E(X)\times Ba(X) \times (0,1]$ with $x \sim (e,y,\lambda)$.

The following is our main theorem of this section. Note that we do not know whether $X_{\mathcal{F}}$ has the $\lambda$-property for every regular family $\mathcal{F}$ and that we have not determined if the space $X_{\mathcal{S}_1}$ has the uniform $\lambda$-property. These remain interesting open questions. 

\begin{theorem}
Let $\alpha$ be a non-zero countable ordinal.
\begin{enumerate}
    \item For $p \in  (1,\infty)$, the space $X^p_{\mathcal{S}_\alpha}$ has the uniform $\lambda$-property.
    \item The space  $X_{\mathcal{S}_\alpha}$ has the $\lambda$-property.
\end{enumerate} 
\label{lambda}
\end{theorem}

\begin{proof}

First, we prove item $(1)$. Let $x \in S(X_{\mathcal{S}_\alpha}^p)$ for $p \in (1,\infty)$. Using Lemma \ref{the big one} (1), we can find $x_1 \in c_{00}$ and $x_1,x_2 \in S(X_{\mathcal{S}_\alpha}^p)$ and so that $x=1/2(x_1+x_2)$. Now apply Lemma \ref{the big one} (2) to find $x_{1,1}$ and $x_{1,2}$ in $c_{00} \cap S(X_{\mathcal{S}_\alpha}^p)$ each with a non-maximal $1$-set so that $x_1= 1/2(x_{1,1} + x_{1,2})$. Finally, we apply Lemma \ref{the big one} (3) to find $x_{1,1,1}$ and $x_{1,1,2}$ in $c_{00} \cap S(X_{\mathcal{S}_\alpha}^p)$ with $x_{1,1}=1/2(x_{1,1,1}+ x_{1,1,2})$ so that $x_{1,1,1}$ has both a non-maximal $1$-set and for each $i \leqslant \max \supp~x_{1,1,1}$ there is an $F \in \mathcal{A}_{x_{1,1,1}}$ with $i \in F$. Theorem \ref{only reflexive} implies that $x_{1,1,1} \in E(X_{\mathcal{S}_\alpha}^p)$. Therefore $X$ has the uniform $\lambda$-property as
$$x=\frac{1}{8}x_{1,1,1} + \frac{1}{8}x_{1,1,2} + \frac{1}{4}x_{1,2}+\frac{1}{2}x_2.$$

We now prove item $(2)$. The beginning of the proof of (2) is the same, however, we are not able to conclude that $x_{1,1,1} \in E(X_{\mathcal{S}_\alpha})$. We do know, however, that $x_{1,1,1}$ is finitely supported with a non-maximal $1$-set. Therefore there is an $n \in \mathbb{N}$ so that $x_{1,1,1} \in \sspan\{e_1,\cdots,e_n\}$.     By Carath\'eodory's Theorem, every point of the unitary ball of an $n$-dimensional normed space is the convex combination of at most $n+1$ many extreme points of the ball. Hence, there are a $d\leqslant n+1$ and extreme points $(y_i)_{i=1}^d$ of $Ba(\sspan\{e_1,\cdots,e_n\})$ so that $$x_{1,1,1}= \sum_{i=1}^d \lambda_i y_i$$
with $\sum_{i=1}^d \lambda_i=1$ and $\lambda_i \geqslant 0$. Note that $\mathcal{A}_{x_{1,1,1}} \subseteq \mathcal{A}_{y_i}$ for each $i \in \{1, \ldots, d\}$.  Indeed, for every $F \in \mathcal{A}_{x_{1,1,1}}$ we have
$$1 = \sum_{j \in F} |x_{1,1,1}(j)| \leqslant \sum_{i=1}^d \lambda_i \sum_{j \in F} |y_i(j)| \leqslant \sum_{i=1}^d \lambda_i = 1.$$ It follows that each $y_i$ is an extreme point of $X_{\mathcal{S}_\alpha}$ as well. Indeed, if $y_i= 1/2(z+w)$ for $z,w\in Ba(X_{\mathcal{S}_\alpha})$,
then $y_i(k)=z(k)=w(k)$ for all $k \leqslant n$ since $y_i$ is in extreme point of $Ba(\sspan\{e_1,\cdots,e_n\})$ and if $z(k)=y_i(k)+\varepsilon$ for some $k > n$, with $\varepsilon>0$ the coordinate $k$ could be added to a non-maximal 1-set of  $x_{1,1,1}$ (and hence, of  $y_i$ )  in order to witness the fact that $\|z\|>1$. This implies that $y_i$ is in $E(X_{\mathcal{S}_\alpha})$ and so $X_{\mathcal{S}_\alpha}$ has the $\lambda$-property.
\end{proof}

\section{Polyhedrality}
\label{section-poly}

A Banach space $X$ is called polyhedral in if the unit ball of every finite dimensional subspace of $X$ is a polytope (i.e. has finitely many extreme points) 
 and it is called a $(V)$-polyhedral space (which is a stronger property \cite{FonfVes-poly}) if 
$$\sup\{f(x): f \in E(X^*) \setminus D(x)\}<1$$
for all $x \in S(X)$ where $D(x) = \{g \in S(X^*): g(x)=1 \} $. 

We will prove in Theorem \ref{th53} that for each countable non-zero $\alpha$, $X_{\mathcal{S}_\alpha}$ is a $(V)$-polyhedral space. Moreover, $Ba(X_{\mathcal{S}_\alpha})$ is the closed convex hull of its extreme points, i.e., $X_{\mathcal{S}_\alpha}$ are solutions to Lindenstrauss' problem \cite{Lin-Klee}, different from the example found by De Bernardi \cite{DeB-Israel}.

In \cite{AronLohSu-PAMS} the authors prove that a space $X$ has the $\lambda$-property if and only if for each $x \in Ba(X)$ there is a sequence of non-negative scalars $(\lambda_i)_{i=1}^\infty$ and a sequence of extreme points $(e_i)_{i=1}^\infty$ with $\sum_{i=1}^\infty \lambda_i = 1$ and $x= \sum_{i=1}^\infty \lambda_i e_i$. This property is called the convex series representation property (CSRP). Therefore we know that for each countable non-zero $\alpha$ the space $X_{\mathcal{S}_\alpha}$ has the  convex series representation property (CSRP) (it also easy to modify the proof of Theorem \ref{lambda} that the space has the $\lambda$-property to verify the CSRP). It follows that for non-zero countable $\alpha$,  $Ba(X_{\mathcal{S}_\alpha})$ is the closed convex hull of its extreme points.
 
In his blog \cite{Go-blog}, Gowers states (but does not prove) that for regular family of finite sets $\mathcal{F}$ containing the singletons, the set of extreme points  $Ba(X_{\mathcal{F}}^*)$ are elements of the form $\sum_{i \in F} \pm e^*_i$ where $F \in \mathcal{F}^{MAX}$. We use this characterization to prove that $X_{\mathcal{S}_\alpha}$ is (V)-polyhedral for each countable $\alpha$. The first step in establishing Gowers' claim, however, is to prove the following structure theorem for  $Ba(X^*_{\mathcal{F}})$ which we believe is of independent interest. 

\begin{proposition}
Let $\mathcal{F}$ be a regular family of finite subsets of $\mathbb{N}$ containing the singletons. Then
\begin{equation}
Ba(X^*_\mathcal{F})=\{\sum_{i=1}^\infty \lambda_i f_i : \lambda_i \geqslant 0, \,\sum_{i=1}^\infty \lambda_i \leqslant 1,\,f_i \in W_{\mathcal{F}}\}. \label{the goodness}
\end{equation}
Here $W_{\mathcal{F}}= \{\sum_{i \in F} \pm e^*_i : F \in \mathcal{F}\}$ is the norming set of $X_{\mathcal{F}}$.
\label{inf convex}
\end{proposition}

Before we prove Proposition \ref{inf convex}, we need the following easy lemma  whose proof we include for completeness sake. 

\begin{lemma}
Let $X$ and $Z$ be Banach spaces $i:X \to Z$ to be an isometry, and $j:X \to i(X)$ defined by $j=i$. Let $x^*\in X^*$. If $z^* \in Z^*$ is a Hahn-Banach extension of $(j^*)^{-1}(x^*)$ then $i^*(z^*)=x^*$
\label{thanks Banach}
\end{lemma}

\begin{proof}
Fix the spaces $X,Y$, the operators $i,j$, and the functionals $x^*$ and $z^*$ as in the statement of the lemma. We wish to show that $i^*z^*(x)=x^*(x)$ for each $x \in X$. Let $x \in X$. Then
$$(i^*z^*)(x)=z^*(ix)=z^*(jx)= ((j^*)^{-1}x^*)(jx)=x^*(x).$$
This is the desired result.
\end{proof}

\begin{proof}[Proof of Proposition \ref{inf convex}]

Let $\mathcal{F}$ be a compact, spreading, and hereditary family of finite subsets of $\mathbb{N}$ containing the singletons and let $X_{\mathcal{F}}$ be the corresponding  combinatorial space. Consider the following compact subset of $\{-1,0,1\}^\mathbb{N}$.
\begin{equation}
    K_{\mathcal{F}}=\{\sigma \in\{-1,0,1\}^\mathbb{N} : \supp\,\sigma \in \mathcal{F}\}.
\end{equation}
Define the isometric embedding $i:X_{\mathcal{F}} \to C(K_{\mathcal{F}})$ by $i(x)(\sigma) = \sum_{k} x(k)\sigma(k)$ and let $j:X_{\mathcal{F}} \to i(X_{\mathcal{F}})$ be defined by $j=i$.

Then $i^*:C(K_{\mathcal{F}})^* \to X_{\mathcal{F}^*}$ is a quotient  map.
Recall that $C(K_{\mathcal{F}})^*$ can be identified with the Radon measures on $K$, $\mathcal{M}(K_{\mathcal{F}})$. Since $K_\mathcal{F}$ is countable, each $\mu \in \mathcal{M}(K_{\mathcal{F}})$ is in the closed span of the Dirac functionals $\delta_\sigma$ (defined by $\delta_\sigma(f)=f(\sigma)$). That is,

\begin{equation}
    \mu = \sum_{\sigma \in K_\mathcal{F}} |\mu(\{\sigma\})| \sign(\mu(\{\sigma\})) \delta_\sigma
\end{equation}
and $\|\mu\|= \sum_{\sigma \in K_\mathcal{F}} |\mu(\{\sigma\})|$.  It is a well-known fact \cite[Exercise 4.1 page 98]{AlbK-book} that $$E(\mathcal{M}(K))=\{\varepsilon \delta_\sigma : \sigma \in K_\mathcal{F}, \varepsilon \in \{-1,1\}\}$$   
which implies that each $\mu \in Ba(X_{\mathcal{F}})$ can be written as a (possibly infinite) convex combination of extreme points (that is, $\mathcal{M}(K_\mathcal{F})$ has the CSRP).

Let $f \in Ba(X^*_\mathcal{F})$ and consider a Hahn-Banach extension $\mu$ of $(j^*)^{-1}(f)$. By Lemma \ref{thanks Banach} we have $i^*(\mu)=f$. Since $\mu \in Ba(\mathcal{M}(K))$ we have  
\begin{equation}
    \mu = \sum_{\sigma \in K_\mathcal{F}} |\mu(\{\sigma\})|\sign(\mu(\{\sigma\}))\delta_\sigma
\end{equation}
and $\sum_{\sigma \in K_\mathcal{F}} |\mu(\{\sigma\})| \leqslant 1$. Note that $i^*(\delta_\sigma)=\sum_{k \in \supp\,\sigma} \sigma(k)e^*_k=:f_\sigma \in W_\mathcal{F}$. Let $\lambda_\sigma = |\mu(\{\sigma\})|$ and $\varepsilon_\sigma = \sign(\mu(\{\sigma\}))$ and observe that
$$f=i^*(\mu)=\sum_{\sigma \in K_\mathcal{F}} |\mu(\{\sigma\})|\sign(\mu(\{\sigma\}))i^*(\delta_\sigma)=\sum_{\sigma \in K_\mathcal{F}} \lambda_\sigma\varepsilon_\sigma f_\sigma.$$
As $K_{\mathcal{F}}$ is countable, this proves the desired equality.
\end{proof}


\begin{proposition}
\label{prop54}
Let $\alpha$ be a countable ordinal. Then
\begin{enumerate}
    \item $E(X^*_{\mathcal{S}_\alpha})=\{\sum_{i \in F} \varepsilon_i e^*_i : F \in \mathcal{S}_\alpha^{MAX},\varepsilon_i \in \{\pm 1\}\}$
    \item $X^*_{\mathcal{S}_\alpha}$ has the $\lambda$-property.
\end{enumerate}
\end{proposition}

\begin{proof}
Let $\alpha$ be a countable ordinal and $f \in E(X^*_{\mathcal{S}_\alpha})$. Suppose that $f \not \in \{ \sum_{i \in F} \varepsilon_i e^*_i: F \in \mathcal{S}_\alpha^{MAX}\mbox{ and }\varepsilon_i \in \{-1,1\}\}$. We will consider two cases. First we assume that $f \in W_{\mathcal{S}_\alpha}$, then $f = \sum_{i \in F} \varepsilon_i e_i^*$, with $F \in \mathcal{S}_\alpha \setminus \mathcal{S}_\alpha^{MAX}$ and $\varepsilon_i \in \{\pm 1\}$ for every $i \in F$. Let $i_0 \in \nn \setminus F$ such that $F \cup \{i_0\} \in \mathcal{S}_\alpha$. Then,
    $$f = \dfrac{1}{2} [(f+e_{i_0}^*) + (f-e_{i_0}^*)].$$
   This contradicts $f \in E(X^*_{\mathcal{S}_\alpha})$, since $f \pm e_{i_0}^* \in Ba(X_{\mathcal{S}_\alpha}^*)$.

Now assume that $f = \sum_{i=1}^\infty \lambda_i f_i$ (non-trivially), with $\lambda_i \geqslant 0$,  $\sum_{i=1}^\infty \lambda_i \leqslant 1$, and $f_i \in W_{\mathcal{S}_\alpha}$. Let $\lambda = \lambda_1$ and $\lambda_1 \in (0,1)$. Then,
$$f = \lambda f_1 + (1-\lambda)\bigg(\dfrac{\lambda_2}{1-\lambda}f_2 + \dfrac{\lambda_3}{1-\lambda} f_3 + \dots\bigg)$$
and $\bigg(\dfrac{\lambda_2}{1-\lambda}f_2 + \dfrac{\lambda_3}{1-\lambda} f_3 + \dots\bigg) \in Ba(X_{\mathcal{S}_\alpha}^*)$, since $\sum_{i=2}^\infty \dfrac{\lambda_i}{1-\lambda} = 1$. Hence, in this case we also have $f \not \in E(X_{\mathcal{S}_\alpha}^*)$. This proves that $E(X^*_{\mathcal{S}_\alpha}) \subseteq \{ \sum_{i \in F} \varepsilon_i e^*_i: F \in \mathcal{S}_\alpha^{MAX}\mbox{ and }\varepsilon_i \in \{-1,1\}\}.$

On the other hand, let $f = \sum_{i \in F} \varepsilon_i e^*_i$ with $F \in \mathcal{S}_\alpha^{MAX}$ and $\varepsilon_i \in \{-1,1\}$. Suppose that $f \not \in E(X^*_{\mathcal{S}_\alpha})$. Let $g, h \in S(X_\alpha^*)$ such that $f = \dfrac{g+h}{2}$. We claim that
$$g(e_i) = h(e_i) = f(e_i), \mbox{ for every } i \in F.$$
In fact, if we had, for example, $\varepsilon_i = 1$ and $g(e_i) > h(e_i)$ for some $i \in F$, then $g(e_i) = 1+ \eta$, $\eta > 0$, and hence $g \not \in S(X_{\mathcal{S}_\alpha}^*)$.

Suppose now that $g(e_{i_0}) \neq 0$ for some $i_0 \not \in F$. Let $x = \sum_{i \in F} a_i e_i$ such that $f(x) = \sum_{i \in F} |a_i| = 1$ and $|a_i| \neq 0$ for every $i \in F$. Let $\eta = \min \{|a_i|: i \in F\}$ and let $y = x + \dfrac{\eta}{2}e_{i_0}$. Notice that $f(y) = f(x) = 1$ and $\|y\| \geqslant 1$. In fact, we will show that $\|y\| = 1$. To prove this, let $G \in \mathcal{S}_\alpha$.

In the first case, if $i_0 \in G$, then $\sum_{i \in G} |y(i)| \leqslant \sum_{i \in G \cap F} |a_i| + |y(i_0)|$. However, $G \cap F \subsetneq F$, because otherwise we would have $F \cup \{i_0\} \in \mathcal{S}_\alpha$. Thus,
$$\sum_{i \in G} |y(i)| < \sum_{i \in G \cap F} |a_i| + |y(i_0)| < 1 - \eta + \dfrac{\eta}{2} = 1 - \dfrac{\eta}{2}.$$
In the second case if $i_0 \not \in G$, then $\sum_{i \in G} |y(i)| \leqslant \sum_{i \in F} |a_i| = 1$. 

Hence, $\|y\| \leqslant 1$ which implies that $\|y\| = 1$. However,
$$g(y) \geqslant f(x) + g(e_{i_0}) = 1 + \dfrac{\eta}{2},$$
which contradicts the fact that $\|g\| = 1$. Therefore, $f \in E(X_{\mathcal{S}_\alpha}^*)$.

For the proof of item (2) we will show that $X_{\mathcal{S}_\alpha}$ has the CSRP (which, as we noted, is equivalent to having the $\lambda$-property). First note the following: Suppose that $f \in W_{\mathcal{S}_\alpha}$ and $f=\sum_{i \in F} \varepsilon_i e_i^*$ for some $F \in \mathcal{S}_\alpha \setminus \mathcal{S}^{MAX}_\alpha$ and unimodular scalars $\varepsilon_i$. That is $f \in W_{\mathcal{S}_\alpha} \setminus E(X^*_{\mathcal{S}_\alpha})$. Then we can find non-empty set $G$ so that $F\cap G =\emptyset$ and $F\cup G \in \mathcal{S}^{MAX}_\alpha$. 
Let 
$$ f_1 = f + \sum_{i \in G} e^*_i \mbox{ and } f_2 = f - \sum_{i \in G} e^*_i$$
Then $f_1,f_2 \in E(X^*_{\mathcal{S}_\alpha})$ and $f = \frac{1}{2}(f_1 + f_2)$. Therefore since each $f \in Ba(X_{\mathcal{S}_\alpha})$ can be written as an infinite convex combination $f=\sum_{i=1}^\infty \lambda_i g_i$ for $g_i \in W_{\mathcal{S}_\alpha}$ and $g_i = \frac{1}{2}(g_{i,1}+g_{i,2})$ we have
$$ f=\frac{1}{2}\sum_{i=1}^\infty \lambda_i g_{i,1} + \frac{1}{2}\sum_{i=1}^\infty \lambda_i g_{i,2}$$
with $g_{i,j} \in E(X_{\mathcal{S}_\alpha})$ for $i \in \mathbb{N}$ and $j=1,2$. This is the desired result.
\end{proof}

\begin{remark}
The analogous proposition replacing $\mathcal{S}_\alpha$ with a regular family $\mathcal{F}$ also holds with only minor changes to the proof. We choose not to consider this level of generality in order to say consistent with the main objectives of the current paper.
\end{remark}

Finally we can show that $X_{\mathcal{S}_\alpha}$ is (V)-polyhedral space.

\begin{theorem}
\label{th53}
For each countable $\alpha$ the space $X_{\mathcal{S}_\alpha}$ is a $(V)$-polyhedral space. 
\end{theorem}

\begin{proof}
Let $x \in S(X_{\mathcal{S}_\alpha})$ and $f \in E(X^*_{\mathcal{S}_\alpha})$ such that $f(x) < 1$. By Proposition \ref{prop54}, there exists $F \in \mathcal{S}_\alpha^{MAX}$ such that $f = \sum_{i \in F} \varepsilon_i e_i^*$, with $\varepsilon_i \in \{\pm 1\}$ for each $i \in F$. Let $G = \{i \in F: \varepsilon_i = \sign(x(i))\}$ and $H = \{i \in F: \varepsilon_i = - \sign(x(i))\}$.

Notice that $G$ is not a 1-set for $x$. In fact, if $H = \emptyset$, then $f(x) = \sum_{i \in G} |x(i)| < 1$. On the other hand, if $H \neq \emptyset$, then
$$\sum_{i \in G} |x(i)| <  \sum_{i \in F} |x(i)| \leqslant \|x\| = 1.$$
By Proposition \ref{lots of items}, there exists $\varepsilon_x >0 $ such that 
$\sum_{i \in G} |x(i)| \leqslant 1 - \varepsilon_x$. Hence,
$$f(x) = \sum_{i \in G} |x(i)| - \sum_{i \in H} |x(i)| \leqslant 1 - \varepsilon_x.$$
This is the desired result, since $\varepsilon_x$ depends only on $x$.
\end{proof}

\section{The Isometry Group of $X_{\mathcal{S}_n}$}
\label{section-isometry}

In this section, we will use our previous results concerning extreme points of Schreier space to exhibit the general form of the elements of $\Isom(X_{\mathcal{S}_n})$, with $n\in\mathbb{N}$. We state the main result.

\begin{theorem}
Let $n \in \mathbb{N}$ and $U\in \Isom(X_{\mathcal{S}_n})$. Then $Ue_i=\pm e_i$ for each $i \in \mathbb{N}$
\label{Isom}
\end{theorem}

All the work in the section is  related to the proof of the Theorem \ref{Isom}. Let us fix $n \in\mathbb{N}$, $U\in \Isom(X_{\mathcal{S}_n})$ and the following notation throughout this section: Let $U(e_i)=x_i$ and $U(y_i)=e_i$. 

We first require the following technical lemma.

\begin{lemma}
The following hold:
\begin{itemize}
    \item [(i)] We have $Ue_1 = \pm e_1$.
    \item [(ii)] Let $j \in \mathbb{N}$ with $j \geqslant 2$. Then, $x_j \in c_{00}$, $x_j(1)=0$, and $x_j$ has a non-maximal 1-set. 
    \item [(iii)] Let $m \in \mathbb{N}$ and $j>\max\{\max \supp \,x_i: 1 \leqslant i \leqslant m\}$. Then $\min\supp\, y_j>m$.
\end{itemize}
\label{isometry items}
\end{lemma}

\begin{proof}
Let $X_1$ be the subspace of $X_{\mathcal{S}_n}$ of all vectors having $0$ in the first coordinate. It suffices to show that $U(X_1) = X_1$. Note the following characterization of $X_1$: A subspace $X$ of $X_{\mathcal{S}_n}$ is equal to $X_1$ if and only if $X$ is closed with codimension $1$ and there is a norm-one vector $e \in X_{\mathcal{S}_n}$ so that for each $x \in Ba(X)$, $\|e + x\|=1$.

Let us first see that this characterization holds. The forward direction is trivial using $e=e_1$. For the reverse implication, we assume first that the given vector $e$ has the property there is a $j \in \supp \,e$ with $j \geqslant 2$. Since $X$ has codimension $1$ there is a $k >j$ so that $e_k \in X$. Then since $\{j,k\} \in \mathcal{S}_n$ 
$$1=\|e +e_k \|\geqslant |e(j)|+1$$
which is a contradiction. Therefore $\supp\, e = \{1\}$ and thus $e=\pm e_1$. Consequently, $X=X_1$ 

To show that $U(X_1)=X_1$, it therefore suffices to find the appropriate vector `$e$'.  Let $x \in Ba(X_1)$ and note that 
$$ 1 =\|e_1 + x \|= \|Ue_1 + Ux \|$$
Therefore $Ue_1$ is the required vector `$e$' and, consequently $U(X_1)=X_1$,

We now prove item (ii).  Let $j\geqslant 2$. It is easy to see that $e_1 + e_j \in E(X_{\mathcal{S}_n})$. Therefore $U(e_1 + e_j) = \varepsilon_1 e_1 + x_j \in E(X_{\mathcal{S}_n})$ for some $\varepsilon_1 \in \{-1,1\}$. Using Proposition \ref{lots of items} (3), $\varepsilon_1 e_1 + x_j \in c_{00}$ and thus $x_j \in c_{00}$. 

Since $U$ is an isometry $\|\varepsilon_1 e_1 \pm x_j\|=1$. Then $1 \geqslant |\varepsilon_1+ x_j(1)|$ and $1 \geqslant |\varepsilon_1- x_j(1)|$. 
This can only be in the case if $x_j(1)=0$.

In addition, using Proposition \ref{only reflexive}, $\varepsilon_1 e_1 + x_j$ has a non maximal 1-set $F$ and clearly $1 \not\in F$. Therefore $F\subset \supp \,x_j$ and so is a non-maximal 1-set for $x_j$. This concludes the proof of item (ii).  

Proof of item (iii): We will proceed by induction on $m$. For the base case, using (i) we fix $j>1$. Since $U$ is an isometry, $1=\|\varepsilon_1 e_1 \pm e_j\|=\| e_1 \pm y_j\|$. 
This implies that $1 \geqslant |\varepsilon_1 + y_j(1)|$ and $1 \geqslant |\varepsilon_1 - y_j(1)|$. These cannot simultaneously be true unless $y_j(1)=0$, as desired for the base case.

Let $m \in \mathbb{N}$ and $m \geqslant 2$ and assume that the conclusion holds for all $m'<m$. Fix $j_m>\max\{\max \supp \,x_i: 1 \leqslant i \leqslant m\}$. By the induction hypothesis we know that $\min\supp\,y_{j_m}>m-1$. Therefore it suffices to prove that $y_{j_m}(m)=0$. First note that by item (iii), $x_m$ has a non-maximal 1-set $F$. Therefore $F \cup \{j_m\} \in \mathcal{S}_n$ and so $2=\|x_m \pm e_{j_m}\|$. Therefore $\|e_m \pm y_{j_m}\|=2$. Let $F^+ \in \mathcal{S}_n$ with $\sum_{i \in F^+}|(e_m + y_{j_m})(i)|=2$ and $F^- \in \mathcal{S}_n$ with $\sum_{i \in F^-}|(e_m - y_{j_m})(i)|=2$. Since the norm of both of these vectors is $1$ we know that $m \in F^+\cap F^-$. Therefore
$$2 = |1 + y_{j_m}(m)| + \sum_{i \in F^+}|y_{j_m}(i)|,$$
$$2 = |1 - y_{j_m}(m)| + \sum_{i \in F^-}|y_{j_m}(i)|.$$
If $y_{j_m}(m)\not=0$ then either $|1 + y_{j_m}(m)|$ or $|1 - y_{j_m}(m)|$ is strictly less than $1$. Therefore either $\sum_{i \in F^+}|y_{j_m}(i)|$ or $\sum_{i \in F^-}|y_{j_m}(i)|$ is strictly greater than $1$, which contradicts the fact that $\|y_{j_m}\|=1$. Therefore $\min \supp \, y_{j_m} >m$, as desired.
\end{proof}

For $x,y \in c_{00}$ we write $x<y$ if $\max \supp\,x < \min\supp\,y$ and $k<x$ if $k \leqslant \min \supp \,x$. If $F \subset \mathbb{N}$ we will say that $(z_i)_{i \in F}$ is a block sequence if for $i<j$ in $F$ $z_i < z_j$.

\begin{corollary}
For each $m \in \mathbb{N}$ there is an $d\in \mathbb{N}$ and $m < y_d$ and $k \in \mathbb{N}$ with $y_d <y_k$.
\label{makeblocks}
\end{corollary}

\begin{proof}
Fix $m \in \mathbb{N}$. Using Lemma \ref{isometry items} (iii) we can find $d$ sufficiently large so that $m<y_d$. Applying Lemma \ref{isometry items} (iii) for $\max \supp \,y_d$ we can find $k$ with $y_d < y_k$.
\end{proof}

\begin{proof}[Proof of Theorem \ref{Isom}]
 Fix $k \in \mathbb{N}$. We will prove that $x_k = \pm e_k$. The proof proceeds by induction. The case $k=1$ follows from Lemma \ref{isometry items}(i). Now fix a $k \geqslant 2$ and assume the claim holds for all $i <k$.
By repeated applications of Corollary \ref{makeblocks} we can find a set $F_1 \subset \mathbb{N}$ so that $k < F_1$, $|F_1|=k$, and a block sequence $(y_i)_{i\in F_1}$ with $k < \sum_{i\in F_1}y_i=:z_1$. For notational reasons let $k_0=k$. 

Let $k_1 = \max\supp\, z_1$. Find $F_2 \subset \mathbb{N}$ so that $|F_2|=k_1$, and a block sequence $(y_i)_{i \in F_2}$ with $k_1 < \sum_{i \in F_2} y_i =:z_2$. 

Continuing in this way we can construct and increasing sequence $(k_i)_{i=0}^\infty$ so that for each $i$
$$z_{i+1}=\sum_{j \in F_{i+1}} y_j > k_i$$
with $|F_{i+1}|=k_i$, $k_i < F_{i+1}$ and a block sequence $(y_j)_{j \in F_{i+1}}$.

There is a unique $d(n-1) \in \mathbb{N}\cup\{0\}$ so that  $(k_i)_{i=0}^{d(n-1)} \in \mathcal{S}^{MAX}_{n-1}$ (clearly,  $d(0)=0$ and $d(1)=k_0-1$).
Consider the following two remarks.
\begin{remark}
Let $j > k_0$ and $F:= \cup_{i=1}^{d(n-1)+1}F_i$. We claim that
\begin{equation}
 \{j\}\cup F \in \mathcal{S}_{n}.  
 \label{in there}
\end{equation}
Our tool is Remark \ref{squeeze in}. Let $G_i = \{k_{i},\ldots,2k_{i}-1\}$ for $i \in \mathbb{N}\cup\{0\}$. Then $G_0<G_{1}< \cdots <G_{d(n-1)}$ are in $\mathcal{S}_{1}^{MAX}$ and $G:=\cup_{i=0}^{d(n-1)}G_i \in \mathcal{S}^{MAX}_{n}$ by the definition of $d(n-1)$.

Note that $|F_i|=|G_{i-1}|=k_{i-1}$ (i.e. $|F|=|G|$), $F$ is a spread of $G$, and $\min G=k_0 <\min F$. 
 Therefore we can apply Remark \ref{squeeze in} to conclude that (\ref{in there}) holds.
 \label{too big}
\end{remark}

\begin{remark}
Suppose $G \in \mathcal{S}_n^{MAX}$ has the property that there are sets $G_0< \cdots <G_m$ are in $\mathcal{S}_1^{MAX}$ such that $\min G_i \leqslant k_i$ with $G=\cup_{i=0}^m G_i$. Then $m \leqslant d(n-1)$. Indeed suppose $m>d(n-1)$. Since $(k_i)_{i=0}^{d(n-1)} \in \mathcal{S}_{n-1}^{MAX}$ we know that    $(k_i)_{i=0}^m \not \in \mathcal{S}_{n-1}$. Since $\min G_i \leqslant k_i$ we can conclude that $(\min G_i)_{i=0}^m \not \in \mathcal{S}_{n-1}$. Therefore using Remark \ref{split em up} item 2, we conclude that $G\not \in \mathcal{S}_n^{MAX}$.
\label{too small}
\end{remark}



Note that by definition 
$$U(e_k + \sum_{i=1}^{d(n-1)+1} \sum_{j \in F_i} y_j)= x_k + \sum_{i=1}^{d(n-1)+1} \sum_{j\in F_i} e_j.$$
We will show that if $\max\supp\, x_k \geqslant k+1$ then we have the contradiction:
\begin{enumerate}
    \item $\displaystyle \|x_k + \sum_{i=1}^{d(n-1)+1} \sum_{j\in F_i} e_j\| > \sum_{i =1}^{d(n-1)+1} |F_i| $
    \item $\displaystyle \|e_k + \sum_{i=1}^{d(n-1)+1} \sum_{j \in F_{i}} y_j\| \leqslant \sum_{i=1}^{d(n-1)+1} |F_i| $
\end{enumerate}
First we will prove item $(1)$.

Let $j \in \supp \,x_k$ with $j \geqslant k+1$. Using Remark \ref{too big},
$$F=\{j\}\cup\bigcup_{i=1}^{d(n-1)+1} F_i \in \mathcal{S}_n.$$ 
We may therefore conclude that
$$ \|x_k + \sum_{i=1}^{d+1} \sum_{j\in F_i} e_j\| \geqslant |x_k(j)| + \sum_{i=1}^{d(n-1)+1}|F_i|.$$
This prove the first item.

We will now prove the second item. Fix a $G \in \mathcal{S}^{MAX}_n$ (we may assume without loss of generality that $G$ is maximal).
Then $G=\cup_{i=0}^m G_i$ where $G_0< \cdots< G_m$ are in $\mathcal{S}_1^{MAX}$ and $(\min G_i)_{i=0}^m \in \mathcal{S}_{n-1}^{MAX}$. 

First note that if either $k_0 \not\in G$ or $G \cap \supp\, y_j =\emptyset$ for some $j \in \cup_{i=1}^{d(n-1)+1}F_i$ the desired upper bound follows from counting the vectors whose intersection is non-empty. Note that in total there are $1 + \sum_{i=1}^{d(n-1)+1} |F_i|$ many vectors and so missing any single vector (which, notably, have norm 1) yields the desired upper bound. 

 Therefore we may assume that 
 \begin{equation}
     k_0 \in G\mbox{ and }G \cap \supp\, y_j \not=\emptyset\mbox{ for all }j \in \bigcup_{i=1}^{d(n-1)+1}F_i. \label{they intersect}
 \end{equation}
Therefore $k_0 \in G$ and, in particular, $\min G_0 \leqslant k_0$. Since $G_0 \in \mathcal{S}_1^{MAX}$, $k_0<F_1$ and $|F_1|=k_0$, $G_0\cap \supp\,y_{\max F_1}=\emptyset$. Consequently, $\min G_1 \leqslant \max\supp\,y_{\max F_1}=k_1$. Continuing in this manner we see that $\min G_i \leqslant k_i$ and $G_i\cap \supp\,y_{\max F_{i+1}}=\emptyset$ for each $0\leqslant i \leqslant m$. Therefore by Remark \ref{too small} we may conclude that $m \leqslant d(n-1)$. However,
$$G_{m}\cap \supp\,y_{\max F_{m+1}}=\emptyset $$
and $m \leqslant d(n-1)$ contradicts (\ref{they intersect}) and yields the desired upper bound.

 Therefore we can conclude, as desired, that $\max \supp\, x_k \leqslant k$. By induction we know that $Ue_{j}=\varepsilon_j e_{j}$ for each $j<k$. If $k=2$ we have from  Lemma \ref{isometry items}(i) that $x_k(1)=0$ and thus $x_k=\pm e_k$. Suppose $k\geqslant 3$ and let $j<k$. If $j=1$, $x_k(j)=0$ by Lemma \ref{isometry items}(ii). Suppose then that $1<j<k$. Then
 $$2=\|e_j \pm e_k\|=\|\varepsilon_j e_j \pm x_k\|$$
 Since $Ue_j=\varepsilon_j e_j$. Arguing as in the proof of Lemma \ref{isometry items}(iii), we know that if $\sum_{i \in F^+}|(\varepsilon_je_j + x_k)(i)|=2$ for $F^+ \in \mathcal{S}_n$ then $j \in F^+$ and if $\sum_{i \in F^-}|(\varepsilon_je_j - x_k)(i)|=2$ for $F^- \in \mathcal{S}_n$ then $j \in F^-$. Therefore
 $$2=|\varepsilon_j + x_k(j)| +\sum_{i \in F^+, i\not= j}|x_k(i)|,$$
 $$2=|\varepsilon_j - x_k(j)| +\sum_{i \in F^-, i\not= j}|x_k(i)|.$$
 Consequently, if $x_k(j)\not=0$ we can see that either $\sum_{\{i \in F^+, i\not= j\}}|x_k(i)|$ or $\sum_{\{i \in F^+, i\not= j\}}|x_k(i)|$ is strictly greater than $1$. This contradicts the fact that $\|x_k\|\leqslant 1$.
 
 Whence $\supp \, x_k =\{k\}$. Since $x_k$ is a norm one vector $x_k = \pm e_k$ which is the desired result. 
 \end{proof}

Using the characterization of the $\Isom(X_{\mathcal{S}_n})$ given by Theorem \ref{Isom} we will provide a new example of a light group of isometries of a Banach space without the PCP.

\begin{proposition}
\label{light}
Let $n \in \mathbb{N}$. The isometry group $\Isom(X_{\mathcal{S}_n})$ is light.
\end{proposition}

\begin{proof}
Let $(T_\alpha)_{\alpha \in I}$ be a net in $\Isom(X_{\mathcal{S}_n})$ such that $T_\alpha \stackrel{\text{WOT}}{\longrightarrow} \Id$ and suppose, by contradiction, that $T_\alpha \stackrel{\text{SOT}}{\centernot \longrightarrow} \Id$. Then, there exist $x \in X_{\mathcal{S}_n}$, $\delta > 0$ and indices $\alpha_1, \alpha_2, \dots \in I$ such that $\|T_{\alpha_\ell}x - x\| > \delta,$ for every $\ell \in \nn$. By Theorem \ref{Isom}, for each $\ell \in \nn$ there exists a sequence $(\varepsilon^{\alpha_\ell}_1,\varepsilon^{\alpha_\ell}_2, \dots)$ in $\{-1,1\}$ such that $T_{\alpha_\ell}e_i=\varepsilon^{\alpha_\ell}_i e_i$ for each $i \in \mathbb{N}.$


Since $T_\alpha \stackrel{\text{WOT}}{\longrightarrow} \Id$, for every $m \in \nn$, $e^*_m(T_{\alpha_\ell}x) \stackrel{\ell \to \infty}{\longrightarrow} x(m)$. Hence, for every $m \in \nn$, there exists $N \in \nn$ such that if $n \geqslant N$, then $(T_{\alpha_\ell}x)(k) = x(k)$, for every $1 \leqslant k \leqslant m$. On the other hand, since $\|T_{\alpha_\ell}x - x \| > \delta$, for every $m \in \nn$ there exists $F_m \in \mathcal{S}_n$ with $\supp(F_m) > m$ such that
$\sum_{k \in F_m} |(T_{\alpha_\ell}x)(k) - x(k)| = \sum_{k \in F_m} 2|x(k)| > \dfrac{\delta}{2}$. Hence, $x$ cannot be approximated by elements of $c_{00}$ with respect to the norm of $X_{\mathcal{S}_n}$, which is a contradiction.
\end{proof}

\def\cprime{$'$} \def\cprime{$'$} \def\cprime{$'$} \def\cprime{$'$}


\begin{thebibliography}{10}

\bibitem{AlbK-book}
F.~Albiac and N.~J. Kalton.
\newblock {\em Topics in {B}anach space theory}, volume 233 of {\em Graduate
  Texts in Mathematics}.
\newblock Springer, New York, 2006.

\bibitem{AlA-Dissertationes}
D.~E. Alspach and S.~A. Argyros.
\newblock Complexity of weakly null sequences.
\newblock {\em Dissertationes Math. (Rozprawy Mat.)}, 321:44, 1992.

\bibitem{AFGR-light}
L.~{Antunes}, V.~{Ferenczi}, S.~{Grivaux}, and C.~{Rosendal}.
\newblock {Light groups of isomorphisms of Banach spaces and invariant LUR
  renormings}.
\newblock {\em arXiv e-prints}, page arXiv:1711.03482, Nov. 2017.

\bibitem{ATol-Memoirs}
S.~A. Argyros and A.~Tolias.
\newblock Methods in the theory of hereditarily indecomposable {B}anach spaces.
\newblock {\em Mem. Amer. Math. Soc.}, 170(806):vi+114, 2004.

\bibitem{AronLoh-Pacific}
R.~M. Aron and R.~H. Lohman.
\newblock A geometric function determined by extreme points of the unit ball of
  a normed space.
\newblock {\em Pacific J. Math.}, 127(2):209--231, 1987.

\bibitem{AronLohSu-PAMS}
R.~M. Aron, R.~H. Lohman, and A.~Su\'arez.
\newblock Rotundity, the {C}.{S}.{R}.{P}., and the {$\lambda$}-property in
  {B}anach spaces.
\newblock {\em Proc. Amer. Math. Soc.}, 111(1):151--155, 1991.

\bibitem{Ba-book}
S.~Banach.
\newblock {\em Th\'{e}orie des op\'{e}rations lin\'{e}aires}.
\newblock \'{E}ditions Jacques Gabay, Sceaux, 1993.
\newblock Reprint of the 1932 original.

\bibitem{BDHQ-preprint}
K.~Beanland, N.~Duncan, M.~Holt, and J.~Quigley.
\newblock Extreme points for combinatorial {B}anach spaces.
\newblock preprint.

\bibitem{Authors-Lambda1}
V.~I. Bogachev, J.~F. Mena-Jurado, and J.~C. Navarro~Pascual.
\newblock Extreme points in spaces of continuous functions.
\newblock {\em Proc. Amer. Math. Soc.}, 123(4):1061--1067, 1995.

\bibitem{CastGon-Extracta}
J.~M.~F. Castillo and M.~Gonz\'alez.
\newblock An approach to {S}chreier's space.
\newblock {\em Extracta Math.}, 6(2-3):166--169, 1991.

\bibitem{Ca-Studia}
R.~M. Causey.
\newblock Concerning the {S}zlenk index.
\newblock {\em Studia Math.}, 236(3):201--244, 2017.

\bibitem{Authors-Lambda2}
J.~Daughtry and B.~Weinstock.
\newblock Extreme points of certain {B}anach spaces related to conditional
  expectations.
\newblock {\em Glasg. Math. J.}, 46(1):29--36, 2004.

\bibitem{DeB-Israel}
C.~A. De~Bernardi.
\newblock Extreme points in polyhedral {B}anach spaces.
\newblock {\em Israel J. Math.}, 220(2):547--557, 2017.

\bibitem{fleming-isometries}
R.~J. Fleming and J.~E. Jamison.
\newblock {\em Isometries on {B}anach spaces: function spaces}, volume 129 of
  {\em Chapman \& Hall/CRC Monographs and Surveys in Pure and Applied
  Mathematics}.
\newblock Chapman \& Hall/CRC, Boca Raton, FL, 2003.

\bibitem{Fonf-polyhedral}
V.~P. Fonf.
\newblock Polyhedral {B}anach spaces.
\newblock {\em Mat. Zametki}, 30(4):627--634, 638, 1981.

\bibitem{FonfVes-poly}
V.~P. Fonf and L.~Vesel\'y.
\newblock Infinite-dimensional polyhedrality.
\newblock {\em Canad. J. Math.}, 56(3):472--494, 2004.

\bibitem{GasLeung-Studia}
I.~Gasparis and D.~H. Leung.
\newblock On the complemented subspaces of the {S}chreier spaces.
\newblock {\em Studia Math.}, 141(3):273--300, 2000.

\bibitem{Go-blog}
W.~Gowers.
\newblock Must an explicitly defined {B}anach space contain $c_0$ or
  $\ell_p$?, Feb. 17, 2009.
\newblock Gowers's Weblog: Mathematics related discussions.

\bibitem{Authors-Lambda5}
A.~S. Granero.
\newblock {$\lambda$}-property in {O}rlicz spaces.
\newblock {\em Bull. Polish Acad. Sci. Math.}, 37(7-12):421--431 (1990), 1989.

\bibitem{Klee-Acta}
V.~Klee.
\newblock Polyhedral sections of convex bodies.
\newblock {\em Acta Math.}, 103:243--267, 1960.

\bibitem{Lin-Lambda}
P.-K. Lin.
\newblock Extreme points of {B}anach lattices related to conditional
  expectations.
\newblock {\em J. Math. Anal. Appl.}, 312(1):138--147, 2005.

\bibitem{Lin-Klee}
J.~Lindenstrauss.
\newblock Notes on {K}lee's paper: ``{P}olyhedral sections of convex bodies''.
\newblock {\em Israel J. Math.}, 4:235--242, 1966.

\bibitem{megrelishvili-light}
M.~G. Megrelishvili.
\newblock Operator topologies and reflexive representability.
\newblock In {\em Nuclear groups and {L}ie groups ({M}adrid, 1999)}, volume~24
  of {\em Res. Exp. Math.}, pages 197--208. Heldermann, Lemgo, 2001.

\bibitem{Authors-Lambda4}
J.~C. Navarro-Pascual and M.~G. Sanchez-Lirola.
\newblock Diameter, extreme points and topology.
\newblock {\em Studia Math.}, 191(3):203--209, 2009.

\bibitem{Rosenthal-Handbook}
H.~P. Rosenthal.
\newblock The {B}anach spaces {$C(K)$}.
\newblock In {\em Handbook of the geometry of {B}anach spaces, {V}ol. 2}, pages
  1547--1602. North-Holland, Amsterdam, 2003.

\bibitem{ShTr-Glasgow}
T.~J. Shura and D.~Trautman.
\newblock The {$\lambda$}-property in {S}chreier's space {$S$} and the
  {L}orentz space {$d(a,1)$}.
\newblock {\em Glasgow Math. J.}, 32(3):277--284, 1990.

\end{thebibliography}
\end{document}